\documentclass[12pt]{amsart}
\usepackage{amsmath}
\usepackage{amssymb}
\usepackage{amsfonts}
\usepackage{amsthm}
\usepackage{enumerate}
\usepackage[mathscr]{eucal}
\usepackage{verbatim}
\usepackage{amsthm}
\usepackage{amscd}
\usepackage{appendix}
\usepackage{tikz}

\numberwithin{equation}{section}

\theoremstyle{plain}
\newtheorem{thm}{Theorem}[section]
\newtheorem{lm}[thm]{Lemma}
\newtheorem{cor}[thm]{Corollary}
\newtheorem{prop}[thm]{Proposition}

\theoremstyle{remark}

\theoremstyle{definition}

\newcommand{\bnu}{\begin{enumerate}}
\newcommand{\enu}{\end{enumerate}}

\newcommand{\q}{\quad}
\newcommand{\qq}{\qquad}

\newcommand{\be}{\beta}
\newcommand{\ga}{\gamma}
\newcommand{\Ga}{\Gamma}

\newcommand{\Om}{\Omega}

\newcommand{\la}{\lambda}

\newcommand{\ep}{\epsilon}

\newcommand{\si}{\sigma}

\newcommand{\vp}{\varphi}

\newcommand{\bbz}{\mathbb{Z}}

\newcommand{\bbzn}{\mathbb Z^n}

\newcommand{\bbs}{\mathbb S}

\newcommand{\bbr}{\mathbb{R}}

\newcommand{\bbrn}{\mathbb R^n}
\newcommand{\rn}{\mathbb{R}^n}

\newcommand{\bbn}{\mathbb{N}}

\newcommand{\UU}{\mathcal{U}}

\newcommand{\II}{\mathcal{I}}
\newcommand{\JJ}{\mathcal{J}}
\newcommand{\TT}{\mathcal{T}}
\newcommand{\LL}{\mathcal{L}}

\newcommand{\kkk}{\vec{{k}}\,}
\newcommand{\xxxi}{\vec{{\xi}}\,}
\newcommand{\xxx}{\vec{{x}}\,}
\newcommand{\yyy}{\vec{{y}}\,}
\newcommand{\zzz}{\vec{{z}}\,}

\newcommand{\GGG}{\vec{{G}}}
\newcommand{\0}{\vec{\boldsymbol{0}}}

\newcommand{\f}{\frac}

\newcommand{\nf}{\infty}

\newcommand{\wh}{\widehat}

\newcounter{question}

\newcommand{\bpf}{\begin{proof}}

\newcommand{\epf}{\end{proof}}

\allowdisplaybreaks

\makeindex         

\begin{document}

\author{Loukas Grafakos}
\address{L. Grafakos, Department of Mathematics, University of Missouri, Columbia, MO 65211, USA} 
\email{grafakosl@missouri.edu}

\author{Danqing He}
\address{D. He, School of Mathematical Sciences,
Fudan University, People's Republic of China}
\email{hedanqing@fudan.edu.cn}

\author{Petr Honz\'ik}
\address{P. Honz\'ik, Department of Mathematics,
Charles University, 116 36 Praha 1, Czech Republic}
\email{honzik@gmail.com}

\author{Bae Jun Park}
\address{B. Park, Department of Mathematics, Sungkyunkwan University, Suwon 16419, Republic of Korea}
\email{bpark43@skku.edu}

\thanks{L. Grafakos would like to acknowledge the support of  the Simons Foundation grant 624733. 
D. He is supported by  National Key R$\&$D Program of China (No. 2021YFA1002500), NNSF of China (No. 11701583, No. 12161141014),  and Natural Science Foundation of Shanghai (No. 22ZR1404900).
B. Park is supported in part by NRF grant 2019R1F1A1044075 and was supported in part by a KIAS Individual Grant MG070001 at the Korea Institute for Advanced Study} 

 \title{On pointwise a.e. convergence of multilinear operators} 

\begin{abstract} 
In this work we obtain the pointwise  almost everywhere 
convergence for two families of multilinear operators: (a) 
truncated homogeneous singular integral  operators 
 associated with $L^q$ functions on the sphere and  (b) 
 lacunary  multiplier operators of limited decay. 
 The a.e. convergence is
deduced from the $L^2\times\cdots\times L^2\to L^{2/m}$ boundedness  
of the associated maximal multilinear operators.   
\end{abstract}

\maketitle


\section{Introduction and Preliminaries}

The pointwise a.e. convergence of sequences of operators is of paramount importance   and has been widely studied in several areas of analysis, such as harmonic analysis, PDE, and ergodic theory.  This area boasts challenging  problems,  
indicatively see     \cite{CRV, C66, DGL, JMZ}, 
and is intimately connected with the boundedness of the associated maximal operators; on this see \cite{Stein}. 
Moreover, techniques and tools employed to study a.e. convergence have led to important developments in harmonic analysis. 

Multilinear harmonic analysis has made significant advances in recent years. The founders of this area are  Coifman and Meyer \cite{CM2} who   realized the 
applicability  of multilinear operators  and  introduced their study  in analysis  in the 
mid 1970s. 
Focusing on  operators that commute with translations, a
fundamental difference between the   multilinear  and the  linear  theory  is the existence of a straightforward characterization of boundedness at an initial point, usually $L^2\to L^2$. The lack of an easy characterization of boundedness at an initial point in the multilinear theory creates difficulties in their study. 
Criteria that get very close to characterization of boundedness have   
recently been obtained by  the first two authors and Slav\'ikov\'a \cite{Gr_He_Sl}
and also  by Kato, Miyachi, and Tomita \cite{Katoetal}  in the  bilinear case.  
These criteria were extended to the general $m$-linear case for $m\ge 2$ by the authors of this article in \cite{Paper1}. This reference also contains initial 
$L^2\times\cdots\times L^2\to L^{2/m}$ estimates for  
rough homogeneous multilinear singular integrals associated with 
$L^q$ functions on the sphere and multilinear multipliers of H\"ormander type.  

 The purpose of this work is to obtain the pointwise a.e. convergence of truncated multilinear homogeneous singular integrals and lacunary multilinear multipliers by establishing   boundedness for their associated maximal operators.

We first   introduce multilinear truncated singular integral operators.
Let $\Omega$ be a integrable function, defined on the sphere $\mathbb{S}^{mn-1}$, satisfying the mean value zero property 
\begin{equation}\label{vanishingmtcondition}
\int_{\mathbb{S}^{mn-1}}\Omega~ d\sigma_{mn-1}=0.
\end{equation} Then we define
\begin{equation*}
K(\yyy):=\frac{\Omega(\yyy')}{|\yyy|^{mn}}, \qquad \vec y \neq 0, 
\end{equation*} 
where $\yyy':=\yyy/|\yyy|\in \mathbb{S}^{mn-1}$, and the 
corresponding truncated multilinear operator $\LL_{\Om}^{(\epsilon)}$ by
$$
\mathcal L^{(\ep)}_{\Om}\big(f_1,\dots,f_m\big)(x):=\int_{(\bbrn)^m\setminus B(0,\epsilon)}{K(\yyy)\prod_{j=1}^{m}f_j(x-y_j)}~d\,\yyy
$$
 acting on Schwartz functions $f_1,\dots,f_m$ on $\bbrn$, where $x\in\bbrn$ and $\yyy:=(y_1,\dots,y_m)\in (\bbrn)^m$.
Moreover, by taking $\epsilon\searrow 0$, we obtain the multilinear homogeneous Calder\'on-Zygmund singular integral operator 
\begin{align}
\mathcal L_{\Om}\big(f_1,\dots,f_m\big)(x)&:=\lim_{\epsilon\searrow 0}\mathcal L^{(\ep)}_{\Om}\big(f_1,\dots,f_m\big)(x) \label{epsilonto0}\\
&=p.v. \int_{(\bbrn)^m}{K(\yyy)\prod_{j=1}^{m}f_j(x-y_j)}~d\,\yyy . \nonumber
\end{align}
This is still well-defined for any Schwartz functions $f_1, \dots,f_m$ on $\bbrn$. Here,   $B(0,\epsilon)$ is  the ball centered at zero with radius $\epsilon>0$ in $(\bbrn)^m$.  
   In \cite{Paper1} we   showed that 
   if $\Om$ lies in $ L^q
(\mathbb{S}^{mn-1}) $ with $q>\f{2m}{m+1}$, then the multilinear singular integral operator $\LL_{\Om}$ admits a  bounded 
extension from $L^2(\bbrn)\times \cdots \times L^2(\bbrn)$ to $L^{2/m}(\bbrn)$. In order words, given $f_j\in L^2(\bbrn)$, $\LL_{\Omega}(f_1,\dots,f_m)$ is well-defined and is in $L^{2/m}(\bbrn)$. It is natural to expect that, similar to the linear case, the truncated operator $\LL_{\Omega}^{(\epsilon)}(f_1,\dots,f_m)$ converges a.e. to $\LL_{\Omega}(f_1,\dots,f_m)$ as $\epsilon\to 0$.



Our first main result is as follows.
\begin{thm}\label{CCC1'}
Let $m\ge 2$, $\frac{2m}{m+1}<q\le \infty$ and $\Omega \in L^q (\mathbb S^{mn-1})$     satisfy \eqref{vanishingmtcondition}.
Then  the  truncated singular integral  
$\LL_{\Om}^{(\epsilon)}(f_1,\dots, f_m)$ converges a.e. as $\epsilon\to 0$ when $f_j\in L^2(\rn)$, $j=1,\dots , m$.
That is, the multilinear singular integral $\LL_{\Omega}(f_1,\dots,f_m)$ is well-defined a.e. when $f_j\in L^2(\rn)$, $j=1,\dots , m$.
\end{thm}

To achieve this goal, we need the following result from  \cite{GDOS}.
\begin{prop}\label{prop00'}  
Let $0<p_j \le\infty$, $1\le j\le m$, $0<q<\infty$ and let $D_j$ be a dense subclass of $L^{p_j}(\bbrn)$. 
Let $\{T_t\}_{t>0}$  be a 
family of $m$-linear operators while $T_\ast$   is the associated maximal operator, defined by
$$T_{\ast}(f_1,\dots,f_m):=\sup_{t>0} \big|T_t(f_1,\dots,f_m)\big| $$
for $f_j\in D_j$, $1\le j\le m$.  Suppose that there is a constant $B$ such that 
\begin{equation}\label{100'}
\|T_{\ast} (f_1,\dots,f_m)\|_{L^{q,\infty}(\bbrn)} \leq B \prod_{j=1}^m \|f_j\|_{L^{p_j}(\bbrn)}
 \end{equation}
 for all $f_j\in D_j(\rn)$. Also suppose  that for all $f_j $ in $D_j$ we have 
\begin{equation}\label{tto1'}
 \lim_{t\to 0} T_{t} (f_1,\dots, f_m) = T(f_1,\dots,f_m)
 \end{equation}
exists and is finite. 
Then for all functions $f_j\in L^{p_j}(\rn) $ the limit in \eqref{tto1'} exists and is finite  a.e., and defines an $m$-linear operator   which uniquely extends $T$ defined on $D_1\times \cdots \times D_m$ and
which is bounded from $L^{p_1}(\bbrn)\times \cdots \times L^{p_m}(\bbrn)$ to $L^{q,\infty}(\rn)$. 
\end{prop}

 

 With the help of this result, we reduce Theorem~\ref{CCC1'} to the boundedness of 
the  associated maximal singular integral operator 
\begin{equation*}
\LL_{\Om}^*\big(f_1,\dots,f_m\big)(x):=\sup_{\epsilon>0}\Big| \LL_{\Omega}^{(\epsilon)}\big(f_1,\dots,f_m\big)(x)\Big|.
\end{equation*}
\begin{thm}\label{MAXSINGINT}
Let $m\ge 2$, $\frac{2m}{m+1}<q\le \infty$ and $\Omega \in L^q (\mathbb S^{mn-1})$     satisfy \eqref{vanishingmtcondition}. Then there exists a constant $C>0$ such that
\begin{equation}\label{maxinequal}
\big\Vert \LL_{\Om}^*(f_1,\dots,f_m)\big\Vert_{L^{2/m}(\bbrn)}\le C\Vert \Om\Vert_{L^q(\mathbb{S}^{mn-1})}\prod_{j=1}^{m}\Vert f_j\Vert_{L^2(\bbrn)}
\end{equation}
for Schwartz functions $f_1,\dots,f_m$ on $\bbrn$.
\end{thm}
This extends and improves a result obtained in  \cite{BH}  when $m=2$ and $q=\infty$.\\

 The essential contribution of this article is to suitably   combine  
Littlewood-Paley techniques and   wavelet decompositions to reduce the boundedness of $\LL_\Om^*$ to   decaying estimates for norms of  maximal operators associated with lattice bumps; see \eqref{2mmaingoal} for the exact formulation. This   result is actually   proved in terms of Plancherel type inequalities, developed in \cite{Paper1} and stated in Proposition~\ref{keyapplication1}.


\hfill




The tools used to establish Theorem~\ref{CCC1'} turn out to be useful in the study of pointwise convergence problem of several related operators. 
As an example
let us take multilinear multipliers with limited decay  to demonstrate our  idea. 

For a smooth function $\si\in\mathscr{C}^{\infty}((\bbrn)^m)$ and $\nu\in\bbz$
let
\begin{equation}\label{defSsinu}
S_{\si}^{\nu}\big(f_1,\dots,f_m\big)(x):=\int_{(\bbrn)^m}{\si(2^{\nu}\xxxi)\Big( \prod_{j=1}^{m}\wh{f_j}(\xi_j)\Big)e^{2\pi i\langle x,\sum_{j=1}^{m}\xi_j \rangle}} ~d\,\xxxi
\end{equation}
for Schwartz functions $f_1,\dots,f_m$ on $\bbrn$, where $\xxxi:=(\xi_1,\dots,\xi_m) \in (\bbrn)^m$.
We are interested in the poinwise convergence of $S^\nu_\si$ when $\nu\to-\infty$.
We pay particular attention to $\si$ satisfying   the limited 
decay property (for some fixed $a$)
$$
\big|\partial^{\beta}\si(\xxxi) \big|\lesssim_{\beta}|\xxxi|^{-a}
$$
for sufficiently many $\be$. Examples of multipliers of this type include $\wh \mu$, the Fourier transform of the spherical measure $\mu$; see  \cite{Ca79, CW78,  Ru} for the corresponding linear results.

The second contribution of this work  is the following result.
\begin{thm}\label{CCC2'}
Let $m\ge 2$ and $a>\frac{(m-1)n}{2}$.
Let 
$\si\in\mathscr{C}^{\infty}((\bbrn)^m)$ satisfy
\begin{equation}\label{givenassumption}
\big|\partial^{\beta}\si(\xxxi) \big|\lesssim_{\beta}|\xxxi|^{-a}
\end{equation}
for all $|\beta|\le  \big[ \frac{(m-1)n}{2} \big]+1$, where $ \left[ r\right]$ denotes the integer part of $r$.
Then for 
$f_j$ in $L^2(\rn)$, $j=1,\dots , m$, the   functions 
$
S_\sigma^\nu(f_1,\dots, f_m) 
$
converge a.e. to $\si(0) f_1 \cdots f_m$ as $\nu \to -\infty$. Additionally, if $\lim_{y\to \infty} 
\si(y)$ exists and and equals $L$, then the functions 
$S_\sigma^\nu(f_1,\dots, f_m) $ converge a.e. to 
$L f_1 \cdots f_m$ as $\nu \to  \infty$. 
\end{thm}


This problem is also reduced to the boundedness of the associated
 $m$-(sub)linear lacunary maximal multiplier operator  defined by: 
\begin{equation*}
\mathscr{M}_{\si}\big(f_1,\dots,f_m\big)(x):=\sup_{\nu \in\bbz}{\big|S_{\si}^{\nu}\big(f_1,\dots,f_m\big)(x)\big|}.
\end{equation*}
$\mathscr{M}_{\si}$ is the so-called multilinear spherical maximal function  when $\si=\wh\mu$, which was studied extensively recently by \cite{AP, Barrionuevo2017, DG, HHY, JL}. In particular a bilinear version of the following theorem was previously obtained in \cite{Gr_He_Ho2020}.


\begin{thm}\label{application4}
Let $m\ge 2$ and $a>\frac{(m-1)n}{2}$.
Let 
 $\si\in\mathscr{C}^{\infty}((\bbrn)^m)$
be as in Theorem~\ref{CCC2'} 
Then there exists a constant $C>0$ such that
\begin{equation*}
\big\Vert \mathscr{M}_{\si}(f_1,\dots,f_m)\big\Vert_{L^{2/m}(\bbrn)}\le C \prod_{j=1}^{m}\Vert f_j\Vert_{L^2(\bbrn)}
\end{equation*}
for Schwartz functions $f_1,\dots,f_m$ on $\bbrn$.
\end{thm}


It follows from Theorems~\ref{MAXSINGINT} and ~\ref{application4} that
$\LL_{\Om}^*$ and $\mathscr{M}_{\si}$ have  unique bounded extensions from $L^2\times \cdots \times L^2$ to $L^{2/m}$ by density.

Let us now sketch the proof of Theorem~\ref{CCC2'}, taking Theorem \ref{application4} temporarily for granted.
We notice that the claimed 
convergence  holds pointwise everywhere  for smooth functions $f_j$ with compact support
by the Lebesgue dominated convergence theorem.
Then the assertions are immediate consequences of Proposition~\ref{prop00'}.\\

 As Theorems \ref{CCC1'} and \ref{CCC2'} follow from Theorems~\ref{MAXSINGINT} and ~\ref{application4}, respectively, 
we   actually focus on the proof of Theorems~\ref{MAXSINGINT} and ~\ref{application4} in the remaining sections.


\section{Preliminary material}\label{11171}


 We adapt some notations and key estimates  from \cite{Paper1}.
 For the sake of independent reading we review the main tools and notation. 
We begin with  certain orthonormal bases of $L^2$ due to Triebel \cite{Tr2010},  that will be of great use in our work. 
The idea is as follows.  
For any fixed $L\in\bbn$ one can construct real-valued compactly supported functions $\psi_F, \psi_M$ in  $\mathscr{C}^L(\bbr)$ satisfying the following properties:
  $\Vert \psi_F\Vert_{L^2(\bbr)}=\Vert \psi_M\Vert_{L^2(\bbr)}=1$, 
  $\int_{\bbr}{x^{\alpha}\psi_M(x)}dx=0$ for all $0\le \alpha \le L$, and moreover, 
if     $\Psi_{\GGG}$ is a function on $\bbr^{mn}$, defined by
$$\Psi_{\GGG}(\xxx):=\psi_{g_1}(x_1)\cdots \psi_{g_{mn}}(x_{mn})$$
for $\xxx:=(x_1,\dots,x_{mn})\in \bbr^{mn}$ and $\GGG:=(g_1,\dots,g_{mn})$ in the set     
 $$\II:=\big\{\GGG:=(g_1,\dots,g_{mn}):g_i\in\{F,M\} \big\},$$
then the family of functions
\begin{equation*}
\bigcup_{\la\in\bbn_0}\bigcup_{\kkk\in \bbz^{mn}}\big\{ 2^{\la{mn/2}}\Psi_{\GGG}(2^{\la}\xxx-\kkk):\GGG\in \II^{\la}\big\}
\end{equation*}
forms an orthonormal basis of $L^2(\bbr^{mn})$,
where $\II^0:=\II$ and  for $\la \ge 1$, we set  $\II^{\la}:=\II\setminus \{(F,\dots,F)\}$.

We   consistently use   the notation $\xxxi:=(\xi_1,\dots,\xi_m) $ for elements of $ (\bbrn)^m$, $\GGG:=(G_1,\dots,G_m)\in (\{F,M\}^n)^m$, and  
$
\Psi_{\GGG}(\xxxi)=\Psi_{G_1}(\xi_1)\cdots \Psi_{G_m}(\xi_m).
$
For each $\kkk:=(k_1,\dots,k_m)\in (\bbzn)^m$ and $\la\in \bbn_0$, let 
$$
\Psi_{G_i,k_i}^{\la}(\xi_i):=2^{\la n/2}\Psi_{G_i}(2^{\la}\xi_i-k_i), \qq 1\le i\le m
$$
and
$$
\Psi_{\GGG,\kkk}^{\la}(\xxxi\,):=\Psi_{G_1,k_1}^{\la}(\xi_1)\cdots \Psi_{G_m,k_m}^{\la}(\xi_m).
$$
We also assume that the  support of $\psi_{g_i}$ is contained in $\{\xi\in \bbr: |\xi|\le C_0 \}$ for some $C_0>1$,
which implies that
\begin{equation*}
\textup{Supp} (\Psi_{G_i,k_i}^\la)\subset \big\{\xi_i\in\bbrn: |2^{\la}\xi_i-k_i|\le C_0\sqrt{n}\big\}.
\end{equation*}
In other words, the support of $\Psi_{G_i,k_i}^\la$ is contained in the ball centered at $2^{-\la}k_i$ and radius $C_0\sqrt{n}2^{-\la}$.
Then we note that for a fixed $\lambda\in\bbn_0$, elements of $\big\{\Psi_{\GGG,\kkk}^{\lambda}\big\}_{\kkk}$ have (almost) disjoint compact supports.





It is also known in \cite{Tr2006} that
if $L$ is sufficiently large, then every tempered distribution $H$ on $\bbr^{mn}$ can be represented as
\begin{equation}\label{daubechewavelet}
H(\xxx)=\sum_{\la\in\bbn_0}\sum_{\GGG\in\II^{\la}}\sum_{\kkk\in \bbz^{mn}}b_{\GGG,\kkk}^{\la}2^{\la mn/2}\Psi_{\GGG}(2^{\la} \xxx -\kkk)
\end{equation}
and for $1<q<\infty$ and $s\ge 0$,
$$
\Big\Vert \Big(\sum_{\vec G,\ \vec k} \big|b_{\vec G,\vec k}^\la \Psi^\la_{\vec G,\vec k}\big|^2\Big)^{1/2} \Big\Vert_{L^q(\bbr^{mn})} \le C2^{-s\la}\|H\|_{L^q_s(\bbr^{mn})} 
$$
where 
\begin{equation*}
b_{\GGG,\kkk}^{\la}:=\int_{\bbr^{mn}}{H(\xxx)\Psi^\la_{\vec G,\vec k}(\xxx)}~d\xxx 
\end{equation*}
and $L^q_s$ is the  Sobolev space of functions $H$ such that 
$(I-\Delta)^{s/2} H \in L^q(\mathbb R^{mn})$. 
Moreover, it follows from the last estimate and 
from the (almost) disjoint support property  of the $\Psi^\la_{\vec G,\vec k}$'s that 
\begin{align}
\big\Vert \big\{b_{\GGG,\kkk}^{\la}\big\}_{\kkk\in \bbz^{mn}}\big\Vert_{\ell^{q}}
\approx&\Big(2^{\la mn(1-q/2)}\int_{\bbr^{mn}}\Big( \sum_{\kkk} \big|b^\la_{\vec G,\vec k}\Psi^\la_{\vec G,\vec k}(\vec x)\big|^2\Big)^{q/2}~ d\xxx\Big)^{1/q}\notag \\
\lesssim& \; 2^{-\la (s-mn/q+mn/2)}\Vert H\Vert_{L^q_s(\bbr^{mn})}. \label{lqestimate}
\end{align}

Now we study an essential estimate in \cite{Paper1} which will play a significant role in the proof of both Theorems \ref{MAXSINGINT} and \ref{application4}.
We define the operator $L_{G_i,k_i}^{\la,\ga}$ by
\begin{equation}\label{lgklg}
L_{G_i,k_i}^{\la,\ga}f:=\big( \Psi_{G_i,k_i}^{\la}(\cdot/2^{\ga})\wh{f}\big)^{\vee}, \qq \ga\in\bbz.
\end{equation} 
For $\mu\in\bbz$ let 
\begin{equation}\label{uset}
\UU^{\mu}:=\big\{ \kkk\in (\bbzn)^m: 2^{\mu-2}\le |\kkk|\le 2^{\mu+2},~  |k_1|\ge\cdots \ge |k_m| \big\}
\end{equation}
and split the set into $m$ disjoint subsets $\UU_l^{\mu}$ as below:
\begin{align*}
\UU_1^{{\mu}}&:=\big\{\kkk\in \UU^{{\mu}}:|k_1|\ge 2C_0\sqrt{n}>|k_2|\ge\cdots \ge |k_m| \big\}\\
\UU_2^{{\mu}}&:=\big\{\kkk\in \UU^{{\mu}}:|k_1|\ge |k_2|\ge 2C_0\sqrt{n}>|k_3|\ge\cdots \ge |k_m| \big\}\\
&\vdots\\
\UU_m^{{\mu}}&:=\big\{\kkk\in \UU^{{\mu}}:|k_1|\ge \cdots\ge |k_m|\ge 2C_0\sqrt{n} \big\}.
\end{align*}
Then we have the following two observations that appear in \cite{Paper1}.
\begin{itemize}
\item For $\kkk\in\UU_l^{\la+\mu}$,
\begin{equation}\label{lgkest}
L_{G_j,k_j}^{\la,\ga}f=L_{G_j,k_j}^{\la,\ga}f^{\la,\ga,{\mu}}  \q \text{ for }~  1\le j\le l 
\end{equation}
due to the support of $\Psi_{G_j,k_j}^{\la}$, where $\wh{f^{\la,\ga,\mu}}(\xi):=\wh{f}(\xi)\chi_{C_0\sqrt{n}2^{\ga-\la}\le |\xi|\le 2^{\ga+\mu+3}}$.
\item For $\mu\ge 1$ and $\la\in\bbn_0$, 
\begin{equation}\label{L2}
\Big( \sum_{\ga\in\bbz}{\big\Vert f^{\la,\ga,{\mu}}\big\Vert_{L^2}^2}\Big)^{1/2}\lesssim ({\mu}+\la)^{1/2}\Vert f\Vert_{L^2} \lesssim \mu^{1/2}(\la+1)^{1/2}\Vert f\Vert_{L^2} 
\end{equation} 
where Plancherel's identity is applied in the first inequality.
\end{itemize}

\begin{prop}[{\cite[Proposition 2.4]{Paper1}}]\label{keyapplication1}
Let $m$ be a positive integer with $m\ge 2$ and $0<q<\frac{2m}{m-1}$. Fix $\la\in\bbn_0$ and $\GGG\in \II^{\la}$.
Suppose that $\{b_{\GGG,\kkk}^{\la,\ga,\mu}\}_{\GGG\in \II^{\la}, \ga,\mu\in \bbz, \kkk\in (\bbzn)^m}$ is a sequence of complex numbers satisfying 
\begin{equation*}
\sup_{\ga \in \bbz}\big\Vert \{b_{\GGG,\kkk}^{\la,\ga,\mu}\}_{\kkk\in (\bbzn)^m} \big\Vert_{\ell^{\infty}}\le A_{\GGG,\la,\mu}
\end{equation*}
 and 
 \begin{equation*}
 \sup_{\ga \in\bbz}\big\Vert \{b_{\GGG,\kkk}^{\la,\ga,\mu}\}_{\kkk\in (\bbzn)^m} \big\Vert_{\ell^{q}}\le B_{\GGG,\la,\mu,q}.
 \end{equation*}
Then the following statements hold:
 \begin{enumerate}
 \item For $1\le r\le 2$,  there exists a constant $C>0$, independent of $,\GGG, \la, \mu$, such that
\begin{align*}
&\Big\Vert \Big(\sum_{\ga\in\bbz}\Big| \sum_{\kkk\in \UU_1^{\la+\mu}} b_{\GGG,\kkk}^{\la,\ga,\mu}L_{G_1,k_1}^{\la,\ga}f_1^{\la,\ga,\mu}\prod
_{j=2}^{m}L_{G_j,k_j}^{\la,\ga}f_j\Big|^r\Big)^{1/r}\Big\Vert_{L^{2/m}}\\
&\le C A_{\GGG,\la,\mu} 2^{\la mn/2}\Big(\sum_{\ga\in\bbz}{\Vert f_1^{\la,\ga,\mu}\Vert_{L^2}^r} \Big)^{1/r} \prod_{j=2}^{m}\Vert f_{j}\Vert_{L^2}
\end{align*}  
for Schwartz functions $f_1,\dots,f_m$ on $\bbrn$.
 \item For $2\le l\le m$ there exists a constant $C>0$, independent of $\GGG, \la, \mu$, such that
\begin{align*}
&\Big\Vert \sum_{\ga\in\bbz} \Big| \sum_{\kkk\in \UU_l^{\la+\mu}} b_{\GGG,\kkk}^{\la,\ga,\mu}\Big( \prod_{j=1}^{l}L_{G_j,k_j}^{\la,\ga}f_j^{\la,\ga,\mu}\Big)\Big( \prod_{j=l+1}^{m}L_{G_j,k_{j}}^{\la,\ga}f_{j}\Big) \Big| \Big\Vert_{L^{2/m}}\\
&\le C A_{\GGG,\la,\mu}^{1-\frac{(l-1)q}{2l}}B_{\GGG,\la,\mu,q}^{\frac{(l-1)q}{2l}} 2^{\la mn/2}\Big[ \prod_{j=1}^{l}\Big(\sum_{\ga\in\bbz}{\Vert f_j^{\la,\ga,\mu}\Vert_{L^2}^2} \Big)^{1/2}\Big] \Big[\prod_{j=l+1}^{m}\Vert f_{j}\Vert_{L^2}\Big]
\end{align*} 
for Schwartz functions $f_1,\dots,f_m$ on $\bbrn$,
 where $\prod_{m+1}^m$ is understood as empty.
 \end{enumerate}
\end{prop}

In view of (\ref{lgkest}), (\ref{L2}) and Proposition \ref{keyapplication1}, we actually obtain
\begin{align}\label{1mainprop}
&\Big\Vert \Big( \sum_{\ga\in\bbz}\Big| \sum_{\kkk\in\UU_{1}^{\la+\mu}}b_{\GGG,\kkk}^{\la,\ga,\mu}\prod_{j=1}^{m}L_{G_j,k_j}^{\la,\ga}f_j\Big|^2\Big)^{1/2}\Big\Vert_{L^{2/m}}\nonumber \\&\lesssim A_{\GGG,\la,\mu}\mu^{1/2}2^{\la mn/2}(\la+1)^{1/2}\prod_{j=1}^{m}\Vert f_j\Vert_{L^2}
\end{align}   
and for $2\le l\le m$
\begin{align}\label{2mainprop}
&\Big\Vert  \sum_{\ga\in\bbz}\Big| \sum_{\kkk\in\UU_{l}^{\la+\mu}}b_{\GGG,\kkk}^{\la,\ga,\mu}\prod_{j=1}^{m}L_{G_j,k_j}^{\la,\ga}f_j\Big| \Big\Vert_{L^{2/m}}\nonumber\\
&\lesssim A_{\GGG,\la,\mu}^{1-\frac{(l-1)q}{2l}}B_{\GGG,\la,\mu,q}^{\frac{(l-1)q}{2l}} \mu^{l/2}2^{\la mn/2}(\la+1)^{l/2}    \prod_{j=1}^{m}\Vert f_{j}\Vert_{L^2}.
\end{align}

\section{An auxiliary  lemma}

We have the following extension of Lemma~5 in \cite{BH}.

\begin{lm}\label{AUX}
Let $1<q\le \infty$ and $\Om \in L^q(\mathbb S^{mn-1})$. Suppose $1<p_1,\dots , p_m<\nf$ and $1/m<p<\infty$ satisfies $1/p=1/p_1+\cdots +1/p_m$ and 
\begin{equation}\label{AUXassume}
\f 1p< \f{1}{q}+\f{m}{q' }.
\end{equation}
 Then the maximal operator
\[
\mathcal M_\Om(f_1,\dots, f_m)(x)=
\sup_{R>0}  \f{1}{R^{mn} } \idotsint\limits_{|\vec y|\le R}   |\Om (\yyy' ) | 
\prod_{j=1}^m \big|  f_j(x-y_j) \big|  ~ d\yyy
\]
maps $L^{p_1}(\mathbb R^n) \times \cdots \times L^{p_m}(\mathbb R^n) $ to 
$L^{p }(\mathbb R^n) $ with norm  bounded by a constant multiple of  $\|\Om\|_{L^q(\mathbb{S}^{mn-1})}$.
\end{lm}

\begin{proof} 
Since $\Vert \Omega\Vert_{L^r(\mathbb{S}^{mn-1})}\lesssim \Vert \Omega\Vert_{L^{\infty}(\mathbb{S}^{mn-1})}$ for all $1<r<\infty$ and there exists $1<q<\infty$ such that $1/p<1/q+m/q'<m~ (=1/\infty+m/1)$,
we may assume $1<q<\infty$. Without loss of generality, we may also assume that $\|\Om\|_{L^q(\mathbb{S}^{mn-1})}=1$. 

We split 
\[
\Om= \Om_0+\sum_{l=1}^\nf \Om_l ,
\]
where $\Om_0=\Om \chi_{|\Om|\le 2}$ and $\Om_l=\Om \chi_{2^l<|\Om|\le 2^{l+1} }$ for $l\ge 1$. 
Then H\"older's inequality and Chebyshev's inequality give
\[
\|\Om_l\|_{L^1}  \le |\textup{Supp} \, \Om_l|^{\f{1}{q'}} \le \|\Om\|_{L^q}^{\frac{q}{q'}} 2^{-l \f{q}{q'}} =2^{-l \f{q}{q'}} 
\]
and obviously 
\begin{equation}\label{Omegainfty}
\|\Om_l \|_{L^\nf} \le 2^{l+1}.
\end{equation}

We first claim that for $1<r,r_1,\dots,r_m<\infty$ with $1/r=1/r_1+\cdots+1/r_m$,
\begin{equation}\label{664455-1}
\big\|\mathcal M_{\Om_l} \big\|_{L^{r_1}\times \cdots \times L^{r_m}\to L^{r}} 
\lesssim \|\Om_l\|_{L^1(\mathbb{S}^{mn-1})}  \lesssim 2^{-l \f{q}{q'}}. 
\end{equation}

To verify this estimate,  
we choose indices $\mu_1,\dots,\mu_m$ satisfying $$1/\mu_1+\cdots+1/\mu_m=1$$ and
$$1<\mu_j<r_j \qq \text{ for each }~ 1\le j\le m. $$
Then a direct computation using H\"older's inequality yields 
\[
\mathcal M_{\Om_l}\big(f_1,\dots, f_m\big)(x) \le  \int_{\mathbb S^{mn-1}}   \big|\Om_l (\vec \theta\,  )\big| \prod_{j=1}^{m}\mathcal{M}_{\mu_j}^{\theta_j}f_j(x)    ~ d\vec \theta\, ,
\] 
where the directional maximal operator $\mathcal{M}_{\mu_j}^{\theta_j}$ is defined by
\[
\mathcal M_{\mu_j}^{\theta_j} f_j(x)  := \sup_{R>0}\Big(  \f{1}{R  }\int_0^R \big|  f (x-t\theta_j) \big|^{\mu_j} ~dt\Big)^{1/\mu_j}. 
\]
It follows from this that 
\[
 \big\| \mathcal M_{\Om_l}(f_1,\dots, f_m)\big\|_{L^{r}} \le  \int_{\mathbb S^{mn-1}}   |\Om (\vec \theta\,  ) |  \prod_{j=1}^m  
\big\| \mathcal M_{\mu_j}^{\theta_j}f_j  \big\|_{L^{r_j}}  ~ d\vec \theta,
\]
where Minkowski's inequality and H\"older's inequality are applied.
Using the $L^{r_j}$ boundedness of $\mathcal M_{\mu_j}^{\theta_j} $ for $0<\mu_j<r_j$ with constants independent of $\theta_j$
(by the method of rotations), we obtain \eqref{664455-1}.

Then the case $p>1$ (for which $q>1$ implies the assumption (\ref{AUXassume})) in the assertion follows from summing the estimates (\ref{664455-1}) over $l\ge 0$.

The other case $1/m<p\le 1$ can be proved by interpolation with the $L^1\times\cdots\times L^1\to L^{1/m,\infty}$ estimate.
Let $\mathcal{M}$ be the Hardy-Littlewood maximal operator.
Then it is easy to see the pointwise estimate
$$\mathcal M_{\Om_l}\big(f_1,\dots , f_m\big)(x) \le 2^{l+1} \prod_{j=1}^m \mathcal M f_j(x).$$
and this proves that
\begin{equation}\label{664455}
\big\|\mathcal M_{\Om_l} \big\|_{L^1\times \cdots \times L^1\to L^{1/m,\nf}} \lesssim 2^l, 
\end{equation}
 using H\"older's inequality for weak type spaces (\cite[p 16]{CFA}), the estimate (\ref{Omegainfty}), and the weak $(1,1)$ boundedness of $\mathcal{M}$.
 Now we fix $0<p_1,\dots,p_m<\infty$ and $1/m<p\le 1$, and choose $r>1$ such that 
\begin{equation*}
\frac{1}{p}<\frac{1}{rq}+\frac{m}{q'} \,\,\,\Big( \!\!<\frac{1}{q}+\frac{m}{q'} \Big),
\end{equation*}
or,  equivalently,
$$\frac{q(m-1/p)}{q'(m-1/r)}-\frac{1/p-1/r}{m-1/r}>0.$$
Then the interpolation between (\ref{664455}) and (\ref{664455-1}) with appropriate $(r_1,\dots, r_m)$ satisfying $1/r=1/r_1+\dots+1/r_m$ (using Theorem  7.2.2 in \cite{MFA}) yields
\begin{equation*}
\big\|\mathcal M_{\Om_l} \big\|_{L^{p_1}\times \cdots \times L^{p_m}\to L^{p }}   \lesssim 2^{l \f{1/p-1/r}{m-1/r}} 2^{-l \f{q}{q'} \f{m-1/p}{m-1/r} }=2^{-l(\frac{q(m-1/p)}{q'(m-1/r)}-\frac{1/p-1/r}{m-1/r})}
\end{equation*}
Finally, the exponential decay in $l$ together 
with the fact that $\|\cdot \|_{L^p}^p$ is a subadditive quantity for $0<p\le 1$ implies the claimed conclusion. 
\end{proof}

\section{Proof of Theorem \ref{MAXSINGINT}}

Let $\frac{2m}{m+1}<q<2$ and $\Om$ in $L^q(\mathbb S^{mn-1})$.
We  use a dyadic decomposition introduced by Duoandikoetxea and Rubio de Francia \cite{Du_Ru}.
We choose a Schwartz function $\Phi^{(m)}$ on $(\bbrn)^m$ such that its Fourier transform $\wh{\Phi^{(m)}}$ is supported in the annulus $\{\xxxi\in (\bbrn)^m: 1/2\le |\xxxi|\le 2\}$ and satisfies
$\sum_{j\in\bbz}\wh{\Phi^{(m)}_j}(\xxxi)=1$ for $\xxxi\not= \0$ where $\wh{\Phi_j^{(m)}}(\xxxi):=\wh{\Phi^{(m)}}(\xxxi/2^j)$.
For $\ga\in\bbz$  let
 $$ K^{\ga}(\yyy):=\wh{\Phi^{(m)}}(2^{\ga}\yyy)K(\yyy), \quad \yyy\in (\bbrn)^m$$
 and then we observe that $K^\ga(\yyy)=2^{\ga mn} K^0(2^\ga \yyy)$.
 For $\mu\in\bbz$ we define
 \begin{equation}\label{kernelcharacter}
K_{\mu}^{\ga}(\yyy):=\Phi_{{\mu}+\ga}^{(m)}\ast K^{\ga}(\yyy)=2^{\ga mn}[\Phi_{{\mu}}^{(m)}\ast K^{0}](2^\ga \yyy).
\end{equation}
It follows from this definition that 
$$
\wh{K^\ga_\mu}(\xxxi)= \wh{\Phi^{(m)}}(2^{-(\mu+\ga)}\xxxi)\wh{K^0}(2^{-\ga}\xxxi)=\wh{K^0_\mu}(2^{-\ga} \xxxi),
$$
which implies that $\wh{K^\ga_\mu}$ is bounded uniformly in $\ga$ while they have almost disjoint supports, so it is natural to add them together as follows:
 $$K_{\mu}(\yyy):=\sum_{\ga\in \bbz}{K_{\mu}^{\ga}(\yyy)}.$$

\subsection{Reduction}
We introduce the maximal  operator
\begin{equation*}
\LL_{\Om}^{\sharp}\big(f_1,\dots,f_m\big)(x) := 
\sup_{\tau\in \mathbb Z} \bigg| \sum_{\gamma<\tau} 
\int_{(\bbrn)^m }{K^{\gamma}(\yyy)\prod_{j=1}^{m}f_j(x-y_j)}~d\,\yyy 
\bigg| 
\end{equation*}
for $x\in \mathbb R^n$. 
Then we claim that
\begin{equation}\label{Est77}
\LL_{\Om}^*\big(f_1,\dots,f_m\big)
\le \mathcal{M}_{\Omega}\big(f_1,\dots,f_m\big)(x)+  \LL_{\Om}^{\sharp}\big(f_1,\dots,f_m\big) 
\end{equation}
To prove \eqref{Est77} we  introduce the notation 
$$K^{(\epsilon)} ( \vec y\,) := K ( \vec y\,)  \chi_{|\vec y\,| \ge \epsilon}, \qq \widetilde{K}^{(\epsilon)} ( \vec y\,) := K ( \vec y\,) \big( 1-  \wh{\Theta^{(m)}} (\vec y/\epsilon) \big),$$
setting $\wh{\Theta^{(m)}}(\yyy):=1-\sum_{\ga\in\bbn}{\wh{\Phi^{(m)}}(\yyy/2^{\ga})}$ so that \[
\textup{Supp}(\wh{\Theta^{(m)}})\subset \{\yyy\in (\bbrn)^m:|\yyy|\le 2\}
\]
 and $\wh{\Theta^{(m)}}(\yyy)=1$ for $|\yyy|\le 1$.

Given $\epsilon>0$ choose $\rho\in \mathbb Z$ such that $2^{\rho}\le \epsilon<2^{\rho+1} $. 
Then we write
\begin{align}
&\bigg| \int_{(\bbrn)^m\setminus B(0,\epsilon)}{K(\yyy)\prod_{j=1}^{m}f_j(x-y_j)}~d\,\yyy \bigg| 
\notag\\
&\qquad\qquad \le 
\bigg| \int_{(\bbrn)^m }{\big( K^{(\epsilon)}(\yyy) - \widetilde{K}^{(2^{\rho})}(\yyy) \big) \prod_{j=1}^{m}f_j(x-y_j)}~d\,\yyy \bigg| \label{651} \\
&\qquad \qquad\qquad \qquad+ \bigg| \int_{(\bbrn)^m}   \widetilde{K}^{(2^{\rho})}(\yyy) \prod_{j=1}^{m}f_j(x-y_j) ~d\,\yyy\bigg| \label{652}.
\end{align}
Term \eqref{652} is clearly less than
\begin{equation*}
\bigg|\sum_{\ga\in\bbz: \ga<-\rho}\int_{(\bbrn)^m}K^{\ga}(\yyy) \prod_{j=1}^{m}f_j(x-y_j) ~d\,\yyy   \bigg|\le \mathcal{L}_{\Omega}^{\sharp}\big(f_1,\dots,f_m\big)(x),
\end{equation*}
while \eqref{651} is controlled by
$
 \mathcal{M}_{\Omega}\big(f_1,\dots,f_m\big)(x)
$
as
\begin{equation*}
\big|K^{(\epsilon)}(\yyy) - \widetilde K^{(2^{\rho})}(\yyy) \big|\lesssim |K(\yyy)| \chi_{|\yyy|\approx 2^{\rho}}\lesssim \frac{|\Omega(\yyy')|}{2^{\rho mn}}\chi_{|\yyy|\lesssim 2^{\rho}}.
\end{equation*}
Thus \eqref{Est77} follows after taking the supremum over all $\epsilon>0$. 

Since  the boundedness of $\mathcal{M}_{\Omega}$ follows from Lemma~\ref{AUX} with the fact that $q>\f{2m}{m+1}$ implies $\f{m}{2} <\f{1}{q}+\f{m}{q'}$, 
matters reduce to the boundedness of $\LL_{\Om}^{\sharp}$. 

For each $\ga\in\bbz$ let $$K_{\mu}:=\sum_{\ga\in\bbz}K_{\mu}^{\ga}.$$
In  the study of multilinear rough singular integral operators $\mathcal{L}_{\Omega}$ in \cite{Paper1} whose kernel is $\sum_{\gamma\in\bbz}K^{\ga}=\sum_{\mu\in\bbz}\sum_{\ga\in\bbz}K_{\mu}^{\ga}=\sum_{\mu\in\bbz}K_{\mu}$,
the part where $\mu$ is less than a constant is relatively simple because the Fourier transform of $K_{\mu}$ satisfies the estimate 
\begin{equation}\label{symbolest}
\big|\partial^{\alpha}\widehat{K_{\mu}}(\xxxi)\big|\lesssim \Vert \Omega\Vert_{L^q(\mathbb{S}^{mn-1})}|\xxxi|^{-|\alpha|}Q(\mu),\qq 1<q\le \infty
\end{equation}
for all multiindices $\alpha$ and $\xxxi\in\bbr^{mn}\setminus \{0\}$, where  $Q(\mu)=2^{(mn-\delta')\mu}$ if $\mu\ge 0$ and $Q(\mu)=2^{\mu(1-\delta')}$ if $\mu<0$ for some $0<\delta'<1/q'$, which is the condition of the Coifman-Meyer multiplier theorem \cite{CM}, \cite[Theorem 7.5.3]{MFA} with constant $\Vert \Omega\Vert_{L^q(\bbs^{mn-1})}Q(\mu)$.
The remaining case when $\mu$ is large enough was handled by using product-type wavelet decompositions.
We expect a similar strategy would work in handling $\LL_{\Om}^{\sharp}$.

To argue strictly, we write
\begin{equation*}
  \LL_{\Om}^{\sharp}\big(f_1,\dots,f_m\big)  
  \le
 \widetilde   \LL_{\Om}^{\sharp}\big(f_1,\dots,f_m\big)  + \sum_{\mu\in\bbz:2^{\mu-10}>C_0\sqrt{mn}} 
  \LL_{\Om,\mu}^{\sharp}\big(f_1,\dots,f_m\big) ,
  \end{equation*}
where we set
$$
 \widetilde   \LL_{\Om}^{\sharp}\big(f_1,\dots,f_m\big) (x) 
 :=
  \sup_{\tau \in \mathbb Z} \bigg| 
  \int_{(\bbrn)^m }{  \sum_{ \gamma<\tau} \sum_{\mu\in\bbz:2^{\mu-10}\le C_0\sqrt{mn}} K^{\gamma}_{\mu}(\yyy)   \prod_{j=1}^{m}f_j(x-y_j)}~d\,\yyy  
  \bigg| 
$$
and
$$
  \LL_{\Om,\mu}^{\sharp}\big(f_1,\dots,f_m\big)(x) 
  := \sup_{\tau\in \mathbb Z} \bigg| 
  \sum_{\gamma<\tau} 
  \int_{(\bbrn)^m }{  K^{\gamma}_{\mu}(\yyy)   \prod_{j=1}^{m}f_j(x-y_j)}~d\,\yyy  
  \bigg| .
  $$
  Then Theorem \ref{MAXSINGINT} follows from the following two propositions:
  \begin{prop}\label{propo1}
  Let $1<p_1,\dots,p_m\le \infty$ and $1/p=1/p_1+\cdots+1/p_m$. Suppose that  $1<q<\infty$ and $\Omega\in L^q(\mathbb{S}^{mn-1})$ with $\int_{\mathbb{S}^{mn-1}}\Omega d\sigma =0$. Then there exists a constant $C>0$ such that
  \begin{equation*}
  \big\Vert \widetilde{\LL}_{\Omega}^{\sharp}(f_1,\dots,f_m)\big\Vert_{L^p}\le C\Vert \Omega\Vert_{L^q(\mathbb{S}^{mn-1})}\prod_{j=1}^{m}\Vert f_j\Vert_{L^{p_j}}
  \end{equation*}
  for Schwartz functions $f_1,\dots,f_m$ on $\bbrn$.
  \end{prop}

  \begin{prop}\label{propo2}
  Let $\frac{2m}{m+1}<q\le \infty$ and $\Omega\in L^q(\mathbb{S}^{mn-1})$ with $\int_{\mathbb{S}^{mn-1}}\Omega d\sigma =0$. 
  Suppose that $\mu\in\bbz$ satisfies $2^{\mu-10}>C_0\sqrt{mn}$. 
  Then there exist $C, \epsilon_0>0$ such that  
\begin{equation*}
\big\|  \LL_{\Om,\mu}^{\sharp}(f_1,\dots,f_m)\big\|_{L^{2/m} } \lesssim  2^{-\epsilon_0 \mu} \| \Om\|_{L^q(\mathbb{S}^{mn-1})}\prod_{j=1}^{m}\Vert f_j\Vert_{L^2}
\end{equation*}
  for Schwartz functions $f_1,\dots,f_m$ on $\bbrn$.
 
  \end{prop}

\subsection{Proof of Proposition \ref{propo1}}

We  decompose $\widetilde{\LL}_{\Omega}^{\sharp}$ further so that the Coifman-Meyer multiplier theorem is involved:
Setting
$$\widetilde{K}(\yyy):=\sum_{\mu\in\bbz: 2^{\mu-10}\le C_0\sqrt{mn}}K_{\mu}(\yyy)=\sum_{\mu\in\bbz: 2^{\mu-10}\le C_0\sqrt{mn}}\;\sum_{\ga\in\bbz}K_{\mu}^{\gamma}(\yyy),$$
$\widetilde{\mathcal{L}}^{\sharp}_{\Omega}\big(f_1,\dots,f_m\big)(x)$ is controlled by the sum of 
\begin{equation*}
T_{\widetilde{K}}^*\big(f_1,\dots,f_m\big)(x):= \sup_{\tau\in\bbz}\Big|\int_{|y|> 2^{-\tau}}\widetilde{K}(\yyy)\prod_{j=1}^{m}f_j(x-y_j)~d\,\yyy \Big|
\end{equation*}
and
\begin{equation*}
\mathfrak{T}_{K}^{**}\big(f_1,\dots,f_m\big)(x):= \sup_{\tau\in\bbz}\Big| \int_{(\bbrn)^m}  {K}^{**}_{\tau}(\yyy)    \prod_{j=1}^{m}f_j(x-y_j)  ~   d\,\yyy   \Big|, 
\end{equation*}
where 
\begin{equation*}
{K}^{**}_{\tau}(\yyy):=\Big(\sum_{\mu\in\bbz: 2^{\mu-10}\le C_0\sqrt{mn}}\; \sum_{\ga<\tau}K_{\mu}^{\ga}(\yyy)\Big) -\widetilde{K}(\yyy)\chi_{|\yyy|> 2^{-\tau}}.  
\end{equation*}


To obtain the boundedness of $T_{\widetilde{K}}^*$, 
we claim that $\widetilde{K}$ is an $m$-linear Calder\'on-Zygmund kernel with constant $C\Vert \Omega\Vert_{L^q(\mathbb{S}^{mn-1})}$ for $1<q< \infty$.
Indeed, it follows from (\ref{symbolest}) that
\begin{equation*}
\big| \partial^{\alpha} \widehat{\widetilde{K}}(\xxxi)\big|\le \sum_{\mu\in\bbz: 2^{\mu-10}\le C_0\sqrt{mn}}\big|\partial^{\alpha}\widehat{K_{\mu}}(\xxxi)\big|\lesssim \Vert \Omega\Vert_{L^q(\mathbb{S}^{mn-1})}|\xxxi|^{-|\alpha|}
\end{equation*} as the sum of $Q(\mu)$ over $\mu$ satisfying $2^{\mu-10}\le C_0\sqrt{mn}$ converges.
Then $\widetilde{K}$ satisfies the size and smoothness conditions for $m$-linear Calder\'on-Zygmund kernel with constant $C\Vert \Omega\Vert_{L^q(\mathbb{S}^{mn-1})}$, as mentioned in the proof of \cite[Proposition 6]{Gr_To}.
Since $\widetilde{K}$ is a Calder\'on-Zygmund kernel, 
Cotlar's inequality in \cite[Theorem 1]{GT-Indiana} yields that $T_{\widetilde{K}}^*$ is bounded on the full range of exponents with constant $C\Vert \Omega\Vert_{L^{q}(\mathbb{S}^{mn-1})}$.

To handle the boundedness of the operator $\mathfrak{T}_K^{**}$, we observe that the kernel ${K}_{\tau}^{**}$ can be written as
\begin{equation*}
{K}_{\tau}^{**}(\yyy)=\sum_{\mu\in\bbz:2^{\mu-10}\le C_0\sqrt{mn} }\; \Big( \sum_{\ga<\tau}K_{\mu}^{\ga}(\yyy)\chi_{|\yyy|\le 2^{-\tau}}-  \sum_{\ga\ge \tau}K_{\mu}^{\gamma}(\yyy)\chi_{|\yyy|>2^{-\tau}}\Big)
\end{equation*}
and thus 
\begin{equation*}
\mathfrak{T}_{K}^{**}\big(f_1,\dots,f_m\big)(x) \le \sup_{\tau\in\bbz}\; \sum_{\mu\in\bbz:2^{\mu-10}\le C_0\sqrt{mn} } \II_{\mu,\tau}(x)+\JJ_{\mu,\tau}(x)
\end{equation*} where
\begin{equation*}
\II_{\mu,\tau}(x):=  \sum_{\ga<\tau}\Big|\int_{|\yyy|<2^{-\tau}} K_{\mu}^{\ga}(\yyy) \prod_{j=1}^{m}f_j(x-y_j) \;d\,\yyy \Big|,
\end{equation*}
\begin{equation*}
\JJ_{\mu,\tau}(x):= \sum_{\ga \ge \tau}\Big|\int_{|\yyy|\ge 2^{-\tau}} K_{\mu}^{\ga}(\yyy) \prod_{j=1}^{m}f_j(x-y_j) \;d\,\yyy \Big|.
\end{equation*}
We claim that there exists $\epsilon>0$ such that
\begin{equation}\label{IJest}
\II_{\mu,\tau}+ \JJ_{\mu,\tau}\lesssim_{C_0,m,n} 2^{\epsilon \mu}\Vert \Omega\Vert_{L^1(\mathbb{S}^{mn-1})}\prod_{j=1}^{m}\mathcal{M}f_j  ~\text{ uniformly in ~} \tau\in\bbz
\end{equation}
for $\mu$ satisfying $2^{\mu-10}\le C_0\sqrt{mn}$, where we recall $\mathcal M$ is the Hardy-Littlewood maximal operator.
Then, using H\"older's inequality and the boundedness of $\mathcal{M}$, we obtain
$$
\big\Vert  \mathfrak{T}_{K}^{**}\big(f_1,\dots,f_m\big) \big\Vert_{L^p}\lesssim \Vert \Omega\Vert_{L^1(\mathbb{S}^{mn-1})} \Big\Vert \prod_{j=1}^{m}\mathcal{M}f_j  \Big\Vert_{L^p}\lesssim \Vert \Omega\Vert_{L^1(\mathbb{S}^{mn-1})}\prod_{j=1}^{m}\Vert f_j\Vert_{L^{p_j}}
$$
 for $1<p_1,\dots,p_m\le \infty$ and $0<p\le \infty$ satisfying $1/p=1/p_1+\dots+1/p_m$ as $\sum_{\mu: 2^{\mu-10}\le C_0\sqrt{mn}}2^{\epsilon \mu}$ converges.
Therefore, let us prove (\ref{IJest}).

Using (\ref{kernelcharacter}), we have
\begin{align*}
\II_{\mu,\tau}(x)&\lesssim \sum_{\ga<\tau}\int_{|\yyy|<2^{-\tau}}\int_{|\zzz|\approx 1}2^{\ga mn}2^{\mu mn}|\Omega(\zzz')|d\zzz \prod_{j=1}^{m}|f_j(x-y_j)|~d\,\yyy\\
 &\lesssim 2^{\mu mn}\Vert \Omega\Vert_{L^1(\mathbb{S}^{mn-1})}\frac{1}{2^{-\tau mn}}\int_{|\yyy|<2^{-\tau}} \prod_{j=1}^{m}|f_j(x-y_j)| ~d\,\yyy\\
 &\lesssim 2^{ \mu mn}\Vert \Omega\Vert_{L^1(\mathbb{S}^{mn-1})}\prod_{j=1}^{m}\mathcal{M}f_j(x)
\end{align*} as desired.

 In addition,
 \begin{align*}
 \JJ_{\mu,\tau}(x)\le \sum_{\ga\ge \tau}\int_{|\yyy|\ge 2^{-\tau}}2^{\ga mn}\Big|\int_{|\zzz|\approx 1}{\Phi_{\mu}(2^{\ga}\yyy-\zzz)\Omega(\zzz')}d\,\zzz \Big|\prod_{j=1}^{m}|f_j(x-y_j)| ~ d\,\yyy.
 \end{align*}
   Since $\Omega$ has vanishing mean, we have
\begin{align*}
&\Big|\int_{|\zzz|\approx 1}{\Phi_{\mu}(2^{\ga}\yyy-\zzz)\,\Omega(\zzz')}\,d\,\zzz \Big| \\
&\lesssim 2^{\mu (mn+1)}\int_{|\zzz|\approx 1}\int_0^1\big| \nabla \Phi (2^{\mu+\ga}\yyy-2^{\mu}t\zzz)\big|\;dt\; |\Omega(\zzz')|~ d\, \zzz. 
\end{align*}
Now we choose a constant $M$ such that $mn<M<mn+1$ and see that
\begin{align*}
\big| \nabla \Phi (2^{\mu+\ga}\yyy-2^{\mu}t\zzz)\big|&\lesssim_M \frac{1}{(1+|2^{\mu+\ga}\yyy-2^{\mu}t\zzz|)^M}\\
&\lesssim_{C_0,m,n,M} \frac{1}{(1+2^{\mu+\ga}|\yyy|)^M}\le \frac{1}{2^{M(\mu+\ga)}}\frac{1}{|\yyy|^{M}}
\end{align*}
as $|\zzz|\approx 1$, $0<t<1$, and $2^{\mu-10}\le C_0\sqrt{mn}$.
This yields that
\begin{align*}
 \JJ_{\mu,\tau}(x)&\lesssim 2^{\mu(mn+1-M)}\Vert \Omega\Vert_{L^1(\mathbb{S}^{mn-1})}\Big(\sum_{\ga\ge \tau}2^{-\ga(M- mn)}\Big)\\
  &\qq \qq \qq \qq \times \int_{|\yyy|\ge 2^{-\tau}}{\frac{1}{|\yyy|^M}\prod_{j=1}^{m}|f_j(x-y_j)|}~  d\, \yyy.
\end{align*}
Since $M>mn$,  the sum over $\ga\ge \tau$ converges to $2^{-\tau(M-mn)}$ and the integral over $|\yyy|\ge 2^{-\tau}$ is estimated by
\begin{align*}
&\sum_{l=0}^{\infty}\int_{2^{-\tau+l}\le |\yyy|<2^{-\tau+l+1}}{\frac{1}{|\yyy|^M}\prod_{j=1}^{m}|f_j(x-y_j)|}~ d\,\yyy\\
&\lesssim 2^{\tau(M-mn)}\sum_{l=0}^{\infty}2^{-l(M-mn)} \Big(\frac{1}{2^{(-\tau+l+1)mn}}\int_{|\yyy|\le 2^{-\tau+l+1}}\prod_{j=1}^{m}|f_j(x-y_j)| ~ d\,\yyy\Big)\\
&\lesssim 2^{\tau(M-mn)}\prod_{j=1}^{m}\mathcal{M}f_j(x).
\end{align*}
Finally, we have
\begin{equation*}
 \JJ_{\mu,\tau} \lesssim 2^{\mu(mn+1-M)}\Vert \Omega\Vert_{L^1(\mathbb{S}^{mn-1})}\prod_{j=1}^{m}\mathcal{M}f_j ,
\end{equation*} which completes the proof of (\ref{IJest}).

\subsection{Proof of Proposition \ref{propo2}}
The proof is based on the wavelet decomposition and the recent developments in  \cite{Paper1}.
Recalling that $\wh{K_{\mu}^0}\in L^{q'}$, we apply the wavelet decomposition (\ref{daubechewavelet})  to write 
\begin{equation*}
\wh{K_{\mu}^0}(\xxxi)=\sum_{\la\in\bbn_0}\sum_{\GGG\in\II^{\la}}\sum_{\kkk\in (\bbzn)^m}b_{\GGG,\kkk}^{\la,\mu}\Psi_{G_1,k_1}^{\la}(\xi_1)\cdots \Psi_{G_m,k_m}^{\la}(\xi_m)
\end{equation*}
where 
\begin{equation*}
b_{\GGG,\kkk}^{\la,\mu}:=\int_{(\bbrn)^m}{\wh{K_{\mu}^0}(\xxxi)\Psi_{\GGG,\kkk}^{\la}(\xxxi)}~ d\,\xxxi.
\end{equation*}
It is known in \cite{Paper1} that for any $0<\delta<1/q'$,
\begin{equation}\label{maininftyest}
\big\Vert \{b_{\GGG,\kkk}^{\la,\mu}\}_{\kkk}\big\Vert_{\ell^{\infty}}\lesssim      2^{-\delta {\mu}}2^{-\la (M+1+mn)} \Vert \Om\Vert_{L^q(\mathbb{S}^{mn-1})}
\end{equation} where $M$ is the number of vanishing moments of $\Psi_{\GGG}$. 
Moreover, it follows from the inequality (\ref{lqestimate}), the Hausdorff-Young inequality, and Young's inequality that
\begin{align}\label{mainlqest}
\big\Vert \{b_{\GGG,\kkk}^{\la,\mu}\}_{\kkk}\big\Vert_{\ell^{q'}}&\lesssim 2^{-\la mn (1/2-1/q')}\Vert \wh{K_{\mu}^0}\Vert_{L^{q'}}\lesssim 2^{-\la mn(1/q-1/2)}\Vert \Omega \Vert_{L^q(\mathbb{S}^{mn-1})}.
\end{align}

Now we may assume that $2^{\la+\mu-2}\le |\vec k|\le 2^{\la+\mu+2}$ due to the compact supports of $\wh{K_{\mu}^0}$ and $\Psi_{\GGG,\kkk}^{\la}$. 
In addition, by symmetry, it suffices to focus on the case $|k_1|\ge\cdots\ge |k_m|$.
Since $\wh{K_{\mu}^\ga}(\xxxi)=\wh{K_{\mu}^0}(\xxxi/2^\ga)$, the boundedness of $ \LL_{\Om,\mu}^{\sharp}$ is reduced to the inequality
\begin{align}\label{2mmainest}
&\bigg\Vert \sup_{\tau\in\bbz}\Big| \sum_{\la\in\bbn_0}\sum_{\GGG\in\II^{\la}} \sum_{\ga\in\bbz: \ga<\tau}\sum_{\kkk\in\UU^{\la+\mu}} b_{\GGG,\kkk}^{\la,\mu}\prod_{j=1}^{m} L_{G_j,k_j}^{\la,\ga}f_j\Big|\bigg\Vert_{L^{2/m}}\nonumber\\
&\lesssim 2^{-\epsilon_0 \mu}\Vert \Omega\Vert_{L^q(\mathbb{S}^{mn-1})}\prod_{j=1}^{m}\Vert f_j\Vert_{L^2}
\end{align} 
where the operators $L_{G_j,k_j}^{\la,\ga}$ and the set $\UU^{\la+\mu}$ are defined as in (\ref{lgkest}) and (\ref{uset}). 
We split $\UU^{\la+{\mu}}$ into $m$ disjoint subsets $\UU_l^{\la+{\mu}}$ ($1\le l\le m$) as before such that
for $k\in \UU^{\la+\mu}_l$ we have 
$$|k_1|\ge\cdots \ge |k_l|\ge 2C_0\sqrt n\ge |k_{l+1}|\ge\cdots\ge |k_m|.$$
Then the left-hand side of (\ref{2mmainest}) is estimated by
\begin{equation*}
\bigg( \sum_{l=1}^{m}\sum_{\la\in\bbn_0}\sum_{\GGG\in\II^{\la}} \bigg\Vert \sup_{\tau\in\bbz}\Big| \sum_{\ga\in\bbz:\ga<\tau} \TT_{\GGG,l}^{\la,\ga,\mu}(f_1,\dots,f_m) \Big|     \bigg\Vert_{L^{2/m}}^{2/m}\bigg)^{m/2}
\end{equation*} 
where $\TT_{\GGG,l}^{\la,\ga,\mu}$ is defined by
\begin{equation*}
\TT_{\GGG,l}^{\la,\ga,\mu}\big(f_1,\dots,f_m\big):=\sum_{\kkk\in\UU_l^{\la+{\mu}}}b_{\GGG,\kkk}^{\la,\mu} \Big(\prod_{j=1}^{m}L_{G_j,k_j}^{\la,\ga}f_j \Big).
\end{equation*}

We claim that for each $1\le l\le m$ there exists $\epsilon_0, M_0>0$ such that
\begin{align}\begin{split}\label{2mmaingoal}
 \bigg\Vert \sup_{\tau\in\bbz}\Big| &\sum_{\ga\in\bbz:\ga<\tau} \TT_{\GGG,l}^{\la,\ga,\mu}(f_1,\dots,f_m) \Big|     \bigg\Vert_{L^{2/m}} \\
&\lesssim 2^{-\epsilon_0 \mu_0}2^{-\la M_0}\Vert \Omega\Vert_{L^q(\mathbb{S}^{mn-1})}\prod_{j=1}^{m}\Vert f_j\Vert_{L^2},
\end{split}\end{align} 
 which concludes (\ref{2mmainest}). Therefore it remains to prove \eqref{2mmaingoal}.

\subsubsection*{Proof of (\ref{2mmaingoal})}
When $2\le l\le m$, we apply (\ref{2mainprop}) with $2<q'<\frac{2m}{m-1}$, 
along with  (\ref{maininftyest}), and (\ref{mainlqest}) to obtain  
\begin{align}
&\bigg\Vert \sup_{\tau\in\bbz}\Big| \sum_{\ga\in\bbz:\ga<\tau}\TT_{\GGG,l}^{\la,\ga,\mu}\big(f_1,\dots,f_m\big) \bigg\Vert_{L^{2/m}}\le \bigg\Vert  \sum_{\ga\in\bbz} \big| \TT_{\GGG,l}^{\la,\ga,\mu}\big(f_1,\dots,f_m\big) \big| \bigg\Vert_{L^{2/m}}\notag\\
&\lesssim \big\Vert \big\{ b_{\GGG,\kkk}^{\la,\mu}\big\}_{\kkk}\big\Vert_{\ell^{\infty}}^{1-\frac{(m-1)q'}{2m}} \big\Vert \big\{ b_{\GGG,\kkk}^{\la,\mu}\big\}_{\kkk}\big\Vert_{\ell^{q'}}^{\frac{(m-1)q'}{2m}} 2^{\la mn/2}(\la+1)^{l/2}{\mu}^{l/2}\prod_{j=1}^{m}\Vert f_j\Vert_{L^2}\notag\\
&\lesssim \Vert \Om\Vert_{L^q(\mathbb{S}^{mn-1})}2^{-\delta {\mu}(1-\frac{(m-1)q'}{2m})}{\mu}^{m/2}2^{-\la C_{M,m,n,q}}(\la+1)^{m/2}\prod_{j=1}^{m}\Vert f_j\Vert_{L^2} ,\notag  
\end{align} 
where 
$$C_{M,m,n,q}:=(M+1+mn)(1-\frac{(m-1)q'}{2m})+mn(1/q-1/2)\frac{(m-1)q'}{2m}-\frac{mn}{2}.$$
Here we used  the fact that $\frac{l-1}{2l}\le \frac{m-1}{2m}$ for $l\le m$.
  Then \eqref{2mmaingoal} follows from choosing $M$ sufficiently large so that $C_{M,m,n,q}>0$ since $1-\frac{(m-1)q'}{2m}>0$.

Now let us prove (\ref{2mmaingoal}) for $l=1$.
In this case, we first see the estimate
\begin{equation}\label{keykeyest}
\Big\Vert \Big(  \sum_{\ga\in\bbz}\big|     \TT_{\GGG,1}^{\la,\ga,\mu}(f_1,\dots,f_m)   \big|^2\Big)^{1/2}\Big\Vert_{L^{2/m}}  \lesssim 2^{-\epsilon_0\mu} 2^{-M_0\la} \Vert \Omega\Vert_{L^q(\mathbb{S}^{mn-1})}\prod_{j=1}^{m}\Vert f_j\Vert_{L^2}
\end{equation}
for some $\epsilon_0, M_0>0$, which can be proved, as in \cite[Section 6]{Paper1}, by using (\ref{1mainprop}) and (\ref{maininftyest}).

Choose a Schwartz function $\Gamma$ on $\bbrn$ whose Fourier transform is supported in the ball $ \{\xi\in\bbrn: |\xi|\le 2\}$ and is equal to $1$ for $|\xi|\le 1$, and define $\Ga_k:=2^{kn}\Ga(2^k\cdot)$ so that
$\textup{Supp}(\wh{\Ga_k})\subset \{\xi\in\bbrn: |\xi|\le 2^{k+1}\}$ and $\wh{\Ga_k}(\xi)=1$ for $|\xi|\le 2^k$.

Since the Fourier transform of $\TT_{\GGG,1}^{\la,\ga,\mu}(f_1,\dots,f_m)$ is supported in the set 
$\big\{  \xi\in\bbrn: 2^{\ga+\mu-5}\le |\xi|\le 2^{\ga+\mu+4}   \big\}$, we can write
\begin{equation*}
\sum_{\ga\in\bbz: \ga<\tau}\TT_{\GGG,1}^{\la,\ga,\mu}(f_1,\dots,f_m)=\Ga_{\tau+\mu+3}\ast \Big( \sum_{\ga\in\bbz: \ga<\tau}\TT_{\GGG,1}^{\la,\ga,\mu}(f_1,\dots,f_m)\Big)
\end{equation*}
and then split the right-hand side into
 \begin{align*}
& \Ga_{\tau+\mu+3}\ast  \Big( \sum_{\ga\in\bbz}\TT_{\GGG,1}^{\la,\ga,\mu}(f_1,\dots,f_m)\Big)-\Ga_{\tau+\mu+3}\ast  \Big( \sum_{ \ga\in\bbz :\ga\ge \tau}\TT_{\GGG,1}^{\la,\ga,\mu}(f_1,\dots,f_m)\Big).
\end{align*}
Due to the Fourier support conditions of $\Gamma_{\tau+\mu+3}$ and $\TT_{\GGG,1}^{\la,\ga,\mu}(f_1,\dots,f_m)$, the sum in the second term can be actually taken over $\tau\le \ga\le \tau+9$.
Therefore, the left-hand side of (\ref{2mmaingoal}) is controlled by the sum of
\begin{equation}\label{DefI}
I:=\bigg\Vert \sup_{\nu\in\bbz}\Big| \Gamma_{\nu}\ast \Big(\sum_{\ga\in\bbz}{  \TT_{\GGG,1}^{\la,\ga,\mu}(f_1,\dots,f_m)   } \Big)\Big|      \bigg\Vert_{L^{2/m}}
\end{equation}
and
\begin{equation}\label{DefII}
II:=\sum_{\ga=0}^{9}\Big\Vert \sup_{\tau\in\bbz}\big| \Ga_{\tau+\mu+3}\ast T_{\GGG,1}^{\la,\tau+\ga,\mu}(f_1,\dots,f_m)\big|\Big\Vert_{L^{2/m}}.
\end{equation}

First of all, when $0\le \gamma\le 9$,
the Fourier supports of both $\Gamma_{\tau+\mu+3}$ and $T_{\GGG,1}^{\lambda,\tau+\gamma,\mu}(f_!,\dots,f_m)$ are $\{\xi\in \bbrn: |\xi|\sim 2^{\tau+\mu}\}$. This implies that
for any $0<r<1$,
\begin{align*}
&\big| \Ga_{\tau+\mu+3}\ast T_{\GGG,1}^{\la,\tau+\ga,\mu}(f_1,\dots,f_m)(x)\big|\\
&\lesssim 2^{(\tau+\mu)(n/r-n)}\Big( \int_{\bbrn}  \big| \Ga_{\tau+\mu+3}(x-y)\big|^r \big| T_{\GGG,1}^{\la,\tau+\ga,\mu}(f_1,\dots,f_m)(y)\big|^r  ~dy\Big)^{1/r}\\
&\lesssim \Big(\mathcal{M}\big(|T_{\GGG,1}^{\la,\tau+\ga,\mu}(f_1,\dots,f_m)|^r\big)(x)\Big)^{1/r}
\end{align*}
where the Nikolskii inequality (see \cite[Proposition 1.3.2]{Tr1983}) is applied in the first inequality.
Setting $0<r<2/m$, and using the maximal inequality for $\mathcal{M}$ and the embedding $\ell^2\hookrightarrow \ell^{\infty}$
we obtain
\begin{align}\label{peetreargument}
II&\lesssim \big\Vert \sup_{\tau\in\bbz}\big|  T_{\GGG,1}^{\la,\tau,\mu}(f_1,\dots,f_m)   \big| \big\Vert_{L^{2/m}}\\
 &\le \Big\Vert \Big( \sum_{\ga\in\bbz}\big|     \TT_{\GGG,1}^{\la,\ga,\mu}(f_1,\dots,f_m)   \big|^2\Big)^{1/2}\Big\Vert_{L^{2/m}}. \nonumber
\end{align}
Then the $L^{2/m}$ norm is bounded by the right-hand side of (\ref{2mmaingoal}), thanks to (\ref{keykeyest}). This completes the estimate for $II$ defined in \eqref{DefII} and 
we turn our attention to $I$ defined in \eqref{DefI}.

 In the sequel we will make use of the following inequality: 
 if $\widehat{g_{\ga}}$ is supported on $\{\xi\in \rn : C^{-1} 2^{\ga+\mu}\leq |\xi|\leq C2^{\ga+\mu}\}$ for some $C>1$ and $\mu\in \bbz$, then
\begin{equation}\label{marshallest}
\Big\Vert \Big\{ \Phi^{(1)}_j\!\ast\! \Big(\sum_{\ga\in\bbz}{g_{\ga}}\!\Big)\!\Big\}_{\! j\in\mathbb{Z}}\Big\Vert_{L^p(\ell^q)}\lesssim_{C} \big\Vert \big\{ g_j\big\}_{j\in\mathbb{Z}}\big\Vert_{L^p(\ell^q)} \q \text{uniformly in }~\mu
\end{equation}  for $0<p<\infty$. The proof of (\ref{marshallest}) is elementary and standard, so it is omitted here; see \cite[(13)]{Gr_He_Ho} and  \cite[Theorem 3.6]{Ya}  for  related arguments.

To obtain the bound of $I$, we note that
\begin{equation*}
I \approx \Big\Vert \sum_{\ga\in\bbz}{     \TT_{\GGG,1}^{\la,\ga,\mu}(f_1,\dots,f_m)       }\Big\Vert_{H^{2/m}}
\end{equation*} 
 where $H^{2/m}$ is the Hardy space. We refer to \cite[Corollary 2.1.8]{MFA} for the above estimate.
Then, using the Littlewood-Paley theory for Hardy space (see for instance \cite[Theorem 2.2.9]{MFA}) and (\ref{marshallest}), there exists a unique polynomial $Q^{\la,\mu,\GGG}(x)$ such that
\begin{align}\label{3mainest}
 \Big\Vert \sum_{\ga\in\bbz}{     \TT_{\GGG,1}^{\la,\ga,\mu}(f_1,\dots,f_m)       }       -Q^{\la,\mu,\GGG} \Big\Vert_{H^{2/m}}\notag &\lesssim  \Big\Vert \Big( \sum_{\ga\in\bbz}\big|     \TT_{\GGG,1}^{\la,\ga,\mu}(f_1,\dots,f_m)   \big|^2\Big)^{1/2}\Big\Vert_{L^{2/m}}\nonumber\\
 & \lesssim 2^{-\epsilon_0\mu} 2^{-M_0\la} \Vert \Omega\Vert_{L^q(\mathbb{S}^{mn-1})}\prod_{j=1}^{m}\Vert f_j\Vert_{L^2}
 \end{align} where (\ref{keykeyest}) is applied.
Furthermore, 
\begin{align*}
&\Big\Vert    \sum_{\ga\in\bbz}{     \TT_{\GGG,1}^{\la,\ga,\mu}(f_1,\dots,f_m)       }  \Big\Vert_{H^{2/m}}\\
&\approx \bigg\Vert  \sup_{\nu\in\bbz}\Big|\Ga_{\nu}\ast\Big(  \sum_{\ga\in\bbz}{     \TT_{\GGG,1}^{\la,\ga,\mu}(f_1,\dots,f_m)       }  \Big)\Big|\bigg\Vert_{L^{2/m}}\\
&=\bigg\Vert  \sup_{\nu\in\bbz}\Big|\Ga_{\nu}\ast\Big(  \sum_{\ga\in\bbz: \ga\le \nu-\mu+5}{     \TT_{\GGG,1}^{\la,\ga,\mu}(f_1,\dots,f_m)       }  \Big)\Big|\bigg\Vert_{L^{2/m}}\\
&\lesssim \bigg\Vert \sup_{\nu\in\bbz} \Big| \sum_{\ga\in\bbz: \ga\le \nu-\mu+5}{      \TT_{\GGG,1}^{\la,\ga,\mu}(f_1,\dots,f_m)     }\Big|\bigg\Vert_{L^{2/m}}\\
&\le \Big\Vert \sum_{\ga\in\bbz}{  \big|    \TT_{\GGG,1}^{\la,\ga,\mu}(f_1,\dots,f_m)  \big|   }\Big\Vert_{L^{2/m}}
\end{align*} 
where the argument that led to (\ref{peetreargument}) is applied in the first inequality.
As we have discussed in \cite[Section 6.1]{Paper1},   this quantity is finite for all Schwartz functions $f_1,\dots,f_m$.
Accordingly, we have 
$$\sum_{\ga\in\bbz}{     \TT_{\GGG,1}^{\la,\ga,\mu}(f_1,\dots,f_m)       }       -Q^{\la,\mu,\GGG}\in H^{2/m}$$
and 
$$\sum_{\ga\in\bbz}{     \TT_{\GGG,1}^{\la,\ga,\mu}(f_1,\dots,f_m)       }      \in H^{2/m},$$
and thus $Q^{\la,\mu,\GGG}=0$.
Now it follows from (\ref{3mainest}) that 
\begin{equation*}
I \lesssim 2^{-\epsilon_0\mu} 2^{-M_0\la} \Vert \Omega\Vert_{L^q(\mathbb{S}^{mn-1})}\prod_{j=1}^{m}\Vert f_j\Vert_{L^2},
\end{equation*}
as expected. This completes the proof of (\ref{2mmaingoal}).

\section{Proof of Theorem \ref{application4}}\label{pfapp4}
Let $\mu_0$ be the smallest integer satisfying $2^{{\mu}_0-3}>C_0\sqrt{mn}$ and
 $$\wh{\Theta^{(m)}_{{\mu}_0-1}}(\xxxi):=1-\sum_{{\mu}={\mu}_0}^{\infty}{\wh{\Phi_{\mu}^{(m)}}(\xxxi)}.$$
Clearly, 
\begin{equation*}
\wh{\Theta^{(m)}_{{\mu}_0-1}}(\xxxi)+\sum_{{\mu}={\mu}_0}^{\infty}{\wh{\Phi_{\mu}^{(m)}}(\xxxi)}=1
\end{equation*}
and thus we can write
\begin{equation*}
\si(\xxxi)=\wh{\Theta_{{\mu}_0-1}^{(m)}}(\xxxi)\si(\xxxi)+\sum_{{\mu}={\mu}_0}^{\infty}{\wh{\Phi_{\mu}^{(m)}}(\xxxi)\si(\xxxi)}=:\si_{{\mu}_0-1}(\xxxi)+\sum_{{\mu}={\mu}_0}^{\infty}{\si_{\mu}(\xxxi)}.
\end{equation*}
Note that $\si_{{\mu}_0-1}$ is a compactly supported smooth function and thus the corresponding maximal multiplier operator $\mathscr{M}_{\si_{{\mu}_0-1}}$, defined by
\begin{align*}
&\mathscr{M}_{\si_{{\mu}_0-1}}\big(f_1,\dots,f_m\big)(x)\\
&:=\sup_{\nu \in\bbz}\Big|\int_{(\bbrn)^m}{\si_{{\mu}_0-1}(2^{\nu} \xxxi)\Big( \prod_{j=1}^{m}\wh{f_j}(\xi_j)\Big) e^{2\pi i\langle x,\sum_{j=1}^{m}\xi_j \rangle}}d\xxxi \Big|,
\end{align*}
is bounded by a constant multiple of 
$
\mathcal{M}f_1(x)\cdots\mathcal{M}f_m(x)
 $ 
 where $\mathcal{M}$ is the Hardy-Littlewood maximal operator on $\bbrn$ as before.
Using H\"older's inequality and the $L^2$-boundedness of $\mathcal{M}$, we can prove
\begin{equation*}
\big\Vert \mathscr{M}_{\si_{{\mu}_0-1}}(f_1,\dots,f_m) \big\Vert_{L^{2/m}}\lesssim \prod_{j=1}^{m}\Vert f_j\Vert_{L^2}.
\end{equation*}

It remains to show that
\begin{equation}\label{lefttoshow}
\Big\Vert \sum_{{\mu}={\mu}_0}^{\infty}\mathscr{M}_{\si_{\mu}}(f_1,\dots,f_m)\Big\Vert_{L^{2/m}}\lesssim \prod_{j=1}^{m}\Vert f_j\Vert_{L^2}.
\end{equation}
Using the decomposition (\ref{daubechewavelet}), write
\begin{equation}\label{sigmajdef}
\si_{\mu}(\xxxi)=\sum_{\la\in\bbn_0}\sum_{\GGG\in\II^{\la}}\sum_{\kkk\in (\bbzn)^m}{b_{\GGG,\kkk}^{\la,\mu}\Psi_{G_1,k_1}^{\la}(\xi_1)\cdots\Psi_{G_m,k_m}^{\la}(\xi_m)}
\end{equation}
 where 
 \begin{equation*}
 b_{\GGG,\kkk}^{\la,\mu}:=\int_{(\bbrn)^m}{\si_{\mu}(\xxxi)\Psi_{\GGG,\kkk}^{\la}(\xxxi)}d\xxxi.
 \end{equation*} 
 
 Let $M:=\Big[ \frac{(m-1)n}{2}\Big]+1$ and choose $1<q<\frac{2m}{m-1}$ such that
 \begin{equation}\label{choiceq}
 \frac{(m-1)n}{2}<\frac{mn}{q}<\min{(a,M)}.
 \end{equation}
In view of (\ref{lqestimate}), we have
 \begin{align}\label{bqpointbound}
 \big\Vert \{b_{\GGG,\kkk}^{\la,\mu}\}_{\kkk\in (\bbzn)^m}\big\Vert_{\ell^q} &\lesssim 2^{-\la (M-mn/q+mn/2)}\Vert \si_{\mu}\Vert_{L_M^q((\bbrn)^m)}\nonumber\\
  & \lesssim 2^{-\la (M-mn/q+mn/2)} 2^{-{\mu}(a-mn/q)}
 \end{align} where the assumption (\ref{givenassumption}) is applied in the last inequality.

We observe that if ${\mu}\ge {\mu}_0$, then $b_{\GGG,\kkk}^{\la,\mu}$ vanishes unless $2^{\la+{\mu}-2}\le |\kkk|\le 2^{\la+{\mu}+2}$
 due to the compact supports of $\si_{\mu}$ and $\Psi_{\GGG,\kkk}^{\la}$, which allows us to replace the sum over $\kkk\in (\bbzn)^m$ in (\ref{sigmajdef}) by the sum over $2^{\la+{\mu}-1}\le |\kkk|\le 2^{\la+{\mu}+1}$. 
Moreover, we may consider only the case $|k_1|\ge \cdots \ge |k_m|$ as in the previous section.
Therefore, in the rest of the section, we assume
\begin{align*}
\si_{\mu}(\xxxi)&=\sum_{\la\in\bbn_0}\sum_{\GGG\in \II^{\la}}\sum_{\kkk\in \UU^{\la+{\mu}}}{    b_{\GGG,\kkk}^{\la,\mu}\Psi_{G_1,k_1}^{\la}(\xi_1)\cdots\Psi_{G_m,k_m}^{\la}(\xi_m) }\\
&=\sum_{l=1}^{m}\sum_{\la\in\bbn_0}\sum_{\GGG\in \II^{\la}}\sum_{\kkk\in \UU_l^{\la+{\mu}}}{    b_{\GGG,\kkk}^{\la,\mu}\Psi_{G_1,k_1}^{\la}(\xi_1)\cdots\Psi_{G_m,k_m}^{\la}(\xi_m) }\\
&=:\sum_{l=1}^{m}\sum_{\la\in\bbn_0}\sum_{\GGG\in\II^{\la}}\si_{\mu,l}^{\la,\GGG}(\xxxi)
\end{align*}
where the sets $\UU^{\la+\mu}$, $\UU_l^{\la+\mu}$ are defined as before.
Then the left-hand side of (\ref{lefttoshow}) can be controlled by
\begin{equation}\label{boundgoal}
\Big(\sum_{l=1}^{m}\sum_{{\mu}={\mu}_0}^{\infty}\sum_{\la\in\bbn_0}\sum_{\GGG\in\II^{\la}}\big\Vert \mathscr{M}_{\si_{{\mu},l}^{\la,\GGG}}(f_1,\dots,f_m)\big\Vert_{L^{2/m}}^{2/m} \Big)^{m/2}.
\end{equation}

Now we claim that 
\begin{align}\label{app4claim}
&\big\Vert \mathscr{M}_{\si_{{\mu},l}^{\la,\GGG}}(f_1,\dots,f_m)\big\Vert_{L^{2/m}}\nonumber\\
&\lesssim 2^{-\la (M-mn/q)}(\la+1)^{l/2}2^{-{\mu}(a-mn/q)}{\mu}^{l/2}\prod_{j=1}^{m}\Vert f_j\Vert_{L^2}.
\end{align}
Then (\ref{boundgoal}) is less than a constant multiple of 
$\prod_{j=1}^{m}\Vert f_j\Vert_{L^2}$ as desired, due to the choice of $q$ in (\ref{choiceq}).

In order to prove (\ref{app4claim}), we use the estimates (\ref{1mainprop}) and (\ref{2mainprop}). 
We first rewrite
\begin{equation*}
\mathscr{M}_{\si_{{\mu},l}^{\la,\GGG}}\big(f_1,\dots,f_m\big)(x)=\sup_{\ga\in\bbz}\Big| \sum_{\kkk\in \UU_l^{\la+{\mu}}}{b_{\GGG,\kkk}^{\la,\mu} \Big( \prod_{j=1}^{m}L_{G_j,k_j}^{\la,\ga}f_i(x)\Big) }\Big|
\end{equation*}
where $L_{G,k}^{\la,\ga}$ is defined as in (\ref{lgklg}).

When $l=1$, applying the embeddings $\ell^2\hookrightarrow \ell^{\infty}$, $\ell^q\hookrightarrow \ell^{\infty}$, and (\ref{1mainprop}), the left-hand side of (\ref{app4claim}) is less than
\begin{align*}
&\bigg\Vert \bigg(\sum_{\ga\in\bbz} \Big| \sum_{\kkk\in\UU_{1}^{\la+{\mu}}}{b_{\GGG,\kkk}^{\la,\mu}   \prod_{j=1}^{m}{L_{G_j,k_j}^{\la,\ga}f_j}             }    \Big|^2 \bigg)^{1/2}\bigg\Vert_{L^{2/m}}\\
&\lesssim \big\Vert \{  b_{\GGG,\kkk}^{\la,\mu}  \}_{\kkk\in (\bbzn)^m}\big\Vert_{\ell^{q}} \mu^{1/2}2^{\la mn/2}(\la+1)^{1/2}\prod_{j=1}^{m}\Vert f_j\Vert_{L^2}\\
&\lesssim 2^{-\la(M-mn/q)}(\la+1)^{1/2} 2^{-{\mu}(a-mn/q)} {\mu}^{1/2}\prod_{j=1}^{m}\Vert f_j\Vert_{L^2}
\end{align*}
where (\ref{bqpointbound}) is applied in the last inequality.

For the case $2\le l\le m$, we can bound the left-hand side of (\ref{app4claim}) by
\begin{align*}
&\bigg\Vert \sum_{\ga\in\bbz}\Big| \sum_{\kkk\in\UU_{l}^{\la+{\mu}}} b_{\GGG,\kkk}^{\la,\mu}  \prod_{j=11}^{m} L_{G_j,k_j}^{\la,\ga}f_j   \Big| \bigg\Vert_{L^{2/m}}\\
&\lesssim \big\Vert \{b_{\GGG,\kkk}^{\la,\mu}\}_{\kkk\in (\bbzn)^m}\big\Vert_{\ell^{q}} {\mu}^{l/2}2^{\la mn/2}(\la+1)^{l/2}\prod_{j=1}^{m}\Vert f_j\Vert_{L^2}.
\end{align*}
Here, we used the inequality (\ref{2mainprop}) and the embedding $\ell^q\hookrightarrow \ell^{\infty}$.
Then the preceding expression is estimated by
\begin{equation*}
2^{-\la (M-mn/q)}(\la+1)^{l/2}2^{-{\mu} (a-mn/q)}{\mu}^{l/2}\prod_{j=1}^{m}\Vert f_j\Vert_{L^2}
\end{equation*}
in view of (\ref{bqpointbound}). This completes the proof of (\ref{app4claim}).

 \section{Concluding remarks}

As of this writing, we are uncertain how to extend  Theorem~\ref{application4} 
in the non-lacunary case. A new ingredient may be necessary to accomplish this. \\

We have addressed the boundedness of several multilinear and maximal multilinear
operators at the initial point $L^2\times \cdots \times L^2\to L^{2/m}$. Our future 
investigation  related to this project has two main directions: 
(a) to extend this initial estimate to many other operators, such as the   
general maximal multipliers considered in \cite{Gr_He_Ho2020, Ru}, and (b) to obtain 
  $L^{p_1}\times \cdots \times L^{p_m}\to L^{p}$ bounds for all of these operators in the 
  largest possible range of exponents possible. Additionally, one could consider the 
  study of related endpoint estimates. We hope to achieve this goal in future 
  publications.



\end{document}